%% file: main.tex
\documentclass[a4paper,12pt,leqno]{amsart}
\input{Preamble}
\usepackage{verbatim}

\begin{document}

\title[]{The proper geometric dimension of the mapping class group of an  orientable surface with punctures}

\author[N. Colin]{Nestor Colin}
\author[R. Jiménez Rolland]{Rita Jiménez Rolland}
\author[P. L. León Álvarez ]{Porfirio L. León Álvarez}
\address{Instituto de Matemáticas, Universidad Nacional Autónoma de México. Oaxaca de Juárez, Oaxaca, México 68000}
\email{rita@im.unam.mx}
\email{ncolin@im.unam.mx}
\email{porfirio.leon@im.unam.mx}

\author[L. J. Sánchez Saldaña]{Luis Jorge S\'anchez Salda\~na}
\address{Departamento de Matemáticas, Facultad de Ciencias, Universidad Nacional Autónoma de México}
\email{luisjorge@ciencias.unam.mx}




\date{}


\begin{abstract}
  We show that the {\it full} mapping class group of any orientable closed surface with punctures  admits a cocompact classifying space for proper actions of dimension equal to its virtual cohomological dimension. This was proved for closed orientable surfaces and for {\it pure} mapping class groups by Aramayona and Martínez Pérez.  
 As a consequence of our result we also obtain the proper geometric dimension of {\it full} spherical braid groups.
\end{abstract}
\maketitle

\section{Introduction}

Let \( S_g \) be an orientable closed connected surface of genus \( g\geq 0 \) and for \( n\geq 1 \) consider \( \{p_1, \ldots, p_n\} \)  a collection of distinguished points in \( S_g\).   The {\it mapping class group} $\Mod_g^n$ is the group of isotopy classes of all orientation-preserving  homeomorphisms  $f:S_g^n\rightarrow S_g^n$, where $S_g^{n}:=S_g-\{p_1, \ldots, p_n\}$. The group $\modgn$ permutes the set $\{p_1, \ldots, p_n\}$ 
and the kernel of such action is the {\it pure mapping class group} $\PMod_g^n$. When $n=0$, we remove it from the notation.

In this paper, we compute the minimal dimension $\gdfin (\Mod_g^n)$, often called the \emph{proper geometric dimension},
of a classifying space $\underline{E}\Mod_g^n$ for proper actions of $\Mod_g^n$. Recall that, for a discrete group $G$, the space $\underline{E}G$ is a contractible $G$-CW-complex on which G acts properly, and such that the fixed point set of a subgroup $H < G$ is contractible if $H$ is finite, and is empty
otherwise.

Since $\Mod_g^n$ is virtually torsion-free, its virtual cohomological dimension $\vcd(\Mod_g^n)$
is a lower bound for $\gdfin(\Mod_g^n)$. Our main result is the following:

\begin{theorem}\label{thm:main}
    For any $g\geq 0$ and $n\geq 1$, the proper geometric dimension of $\Mod_g^n$ is $$\gdfin(\Mod_g^n)=\vcd(\Mod_g^n).$$
    Moreover,  there exists a  cocompact $\underline{E}\Mod_g^n$ of dimension equal to $\vcd(\Mod_g^n)$. 
\end{theorem}

The computation of the proper geometric dimension follows from  \cref{thm:main:genus:zero} for $g=0$, \cref{thm:main:genus:one} for $g=1$, \cref{thm:main:genus:2} for $g=2$, and \cref{thm:main:genus:atleast3} for $g\geq 3$. These theorems are the main results of Sections \ref{sec:zero}, \ref{sec:one}, \ref{sec:two}, and \ref{sec:atleast3}  respectively, which provides a fair overview of the structure of the paper. For $2g+n>2$, the Teichmüller space $\Tgn$ is known to be a model for $\underline E \Mod_g^n$; see for instance \cite[Proposition 2.3]{WolpertJi} and \cite[Section 4.10]{Lu05}. In fact, 
Ji and Wolpert proved in \cite[Theorem 1.2]{WolpertJi} that the {\it truncated} Teichmüller space $\Tgn(\epsilon)$ is a cocompact classifying space for proper actions for $\epsilon$ sufficiently small. Hence, the {\it moreover} part of the main result follows from our computation of $\gdfin(\Mod_g^n)$ and a result of Lück \cite[Theorem 13.19]{Lu89} (see \cref{thm:Luck} below).

\subsection{Known results and filling a  gap in the literature}  In \cite[Section 2]{Ha86} Harer constructed a {\it spine} of Teichmüller space, that is a $\PMod_g^n$-equivariant deformation retraction of $\Tgn$ provided $n>0$. This spine gives a model for $\underbar{E}\PMod_g^n$, which is of minimal dimension for the mapping class group $\Mod_g^1$ of an orientable surface with exactly one puncture; see \cite[Theorem 2.1]{Ha86} and \cite[Introduction]{LaEspinayLos4}.

 On the other hand, Aramayona and Martínez Pérez proved that the mapping class group of any closed orientable surface $\Mod_g$ \cite[Theorem 1.1]{AMP14} admits a cocompact classifying space for proper actions of dimension equal to its virtual cohomological dimension. They use their result and the Birman exact sequence to show that is also the case for the {\it pure}  mapping class group $\PMod^n_g$ \cite[Corollary 1.3]{AMP14}; see also { arXiv:1302.4238v.2} and 
 \cite[Corollary 3.5]{LaEspinayLos4}. 
 
 However, \cite[Theorem 1.1]{AMP14} does not imply the result for the {\it full} mapping class group $\Mod_g^n$ when $n\geq 2$, and \cite[Theorem 1.1]{AMP14} has been misquoted in the literature.  For instance, the existence of a model of minimal dimension for the {\it full} mapping class group is used in the proofs of \cite[Theorem 1 $\&$ Main Theorem]{JPT16}, \cite[Proposition 5.3 $\&$ Theorem 1.4]{Nucinkis:Petrosyan}, \cite[Theorem 1.1]{AJPTN18}, and \cite[Theorem 1.5]{JRLASS24}.  Our \cref{thm:main} fills this {\it gap}, and together with \cite[Theorem 1.1]{AMP14} it gives the necessary ingredients to prove the following result.

\begin{corollary}[Proposition 5.3 of \cite{Nucinkis:Petrosyan}] Let $S$ be a closed (possibly disconnected) surface with possible
a finite number of punctures. Then there is a cocompact model for $\underline{E}\Mod(S)$ of dimension $\gdfin (\Mod(S)) = \vcd (\Mod(S))$.
\end{corollary}

\subsection{Proper geometric dimension of spherical braid groups} 
The {\it full $n$-th strand spherical braid group $B_n(S_0)$} is the fundamental group of the $n$th unordered configuration space of  the sphere $S_0$.  It is the trivial group for $n=1$ and a cyclic group of order $2$ for $n=2$. 
For all $n\geq 1$ these groups are virtually torsion free and their virtual cohomological dimension is $\vcd(B_n(S_0))=\max\{0,n-3\}$ \cite[Theorem 5]{VcdBraid}. For $n\geq 3$ the mapping class group $\Mod_0^n$ is related to the  $B_n(S_0)$ by the following central extension (see for example \cite[Section 2.4]{BraidSurvey})
\[ 1\rightarrow \mathbb{Z}/2\rightarrow B_n(S_0)\rightarrow \Mod(S_0^n)\rightarrow 1.\]
Note that any cocompact model for $\underline E \Mod(S_0^n)$ is, via the above short exact sequence, a cocompact model for $\underline E B_n(S_0)$.  It follows from \cite[Lemma 5.8]{Lu05} that $\gdfin(B_n(S_0))=\gdfin(\Mod(S_0^n)$, and \cref{thm:main} implies:

\begin{corollary} For any $n\geq 1$, 
$\gdfin(B_n(S_0))=\vcd(B_n(S_0))=\max\{0,n-3\}$. Moreover, there exists a cocompact model for $\underline E B_n(S_0)$ of dimension  $\vcd(B_n(S_0))$.
\end{corollary}



\subsection{Overview of the proof and  the paper} In order to prove our main result \cref{thm:main}, we follow the general strategy of \cite{AMP14}, and obtain the proper geometric dimension $\gdfin(\Mod_g^n)$ by computing its algebraic counterpart $\cdfin(\Mod_g^n)$. This computation amounts to verify that for any finite subgroup $F$ of $\Mod_g^n$, the inequatily $\vcd(NF) +\lambda (F)\leq \vcd(\Mod_g^n)$ holds, and use \cite[Theorem 3.3]{AMP14} (see \cref{thm:aramayona:martinezperez} below). Here $NF$ denotes the {\it normalizer} of $F$ in $\Mod_g^n$  and $\lambda (F)$ is the {\it length} of $F$. In \cref{sec:dimensions} we introduce the necessary  definitions and details.

An application of Nielsen realization and \cref{lemma:Maher} due to Maher,  imply that $\vcd(NF)=\vcd(\Mod_{g_F}^{n_F})$ for any finite subgroup $F$ of $\Mod_g^n$. The parameters $g_F$ and $n_F$ are invariants of the orbifold quotient $S_g^n/F$ as described in \cref{mcg}. Most of the present paper deals with either computing or upper-bounding the parameters $n_F$ and $g_F$.  Once computed $g_F$ and $n_F$, we can use Harer's computation  of the $\vcd$ for mapping class groups (\cref{vcd:mcg}) to obtain $\vcd(\Mod_{g_F}^{n_F})$. On the other hand, $\lambda(F)$ can be bounded by the number of factors in the prime decomposition of $|F|$.

 For $g\geq 3$, we obtain \cref{thm:main:genus:atleast3} as a straightforward application of \cite[Proposition 4.4]{AMP14}, which we explain in \cref{sec:atleast3}.  However, for genus $g=0$, $g=1$ and $g=2$, an independent and careful analysis of the finite subgroups of $\Mod_g^n$ is required. This is done in Sections \ref{sec:zero}, \ref{sec:one} and \ref{sec:two}, respectively.\medskip

\noindent{\bf Acknowledgments.} The authors thank Omar Antolín Camarena and Jesús Nuñez Zimbrón for useful discussions. The first author was funded by CONAHCyT through the program \textit{Estancias Posdoctorales por México.} The third author's work was supported by UNAM \textit{Posdoctoral Program (POSDOC)}. All authors are grateful for the financial support of DGAPA-UNAM grant PAPIIT IA106923.

\section{Preliminaries}

\subsection{Proper geometric and cohomological dimensions}\label{sec:dimensions}

There are several notions of dimension defined for a given group $G$. We recall in this section the notions that will be used in this paper.

A model for the {\it classifying space of $G$ for proper actions $\underbar{E}G$} is a $G$-CW-complex $X$  such that the fixed point set $X^H$ of a subgroup $H < G$ is contractible if $H$ is finite, and is empty otherwise. Such a model always exists and is unique up to proper $G$-homotopy. The {\it proper geometric dimension of $G$} is defined to be
$$\gdfin(G)=\{ n\in\mathbb{Z}_{\geq 0}:\  \text{there exists an $n$-dimensional model for }\underbar{E}G\}.$$

On the other hand,  let {\sc Fin} be the collection of finite subgroups of $G$, and consider the restricted orbit category $\mathcal{O}_{\text{\sc Fin}}G$, which has as objects
the homogeneous $G$-spaces $G/H$, $H\in${\sc Fin}, and morphisms are $G$-maps. A $\mathcal{O}_{\text{\sc Fin}}G$-module is a
contravariant functor from $\mathcal{O}_{\text{\sc Fin}}G$ to the category of abelian groups, and a morphism between
two $\mathcal{O}_{\text{\sc{Fin}}}G$-modules is a natural transformation of the underlying functors. 
The category $\mathcal{O}_{\text{\sc{Fin}}}G$-mod of $\mathcal{O}_{\text{\sc{Fin}}}G$-modules is an abelian category with enough
projectives; see for example \cite[p.7]{MV03}. 
The {\it proper cohomological dimension of $G$} is defined to be
\[\cdfin(G)=\min\left\{ n\in\mathbb{Z}_{\geq 0}:\  \exists\text{ a projective resolution of the  $\mathcal{O}_{\text{\sc{Fin}}}G$-module $\mathbb{Z}_{\text{\sc{Fin}}}$ of length $n$}\right\},\]
where $\mathbb{Z}_{\text{\sc{Fin}}}$  is the constant $\mathcal{O}_{\text{\sc{Fin}}}G$-module  given by $\mathbb{Z}_{\text{\sc{Fin}}} (G/H) = \mathbb{Z}$ for all $H\in{\text{\sc{Fin}}}$, and every morphism of $\mathcal{O}_{\text{\sc{Fin}}}G$ goes to the identity function.

The {\it cohomological dimension $\cd(H)$} of a group $H$ is the length of shortest projective resolution, in the category of $H$-modules, for the trivial $H$-module $\mathbb{Z}$. If  $G$ is virtually torsion free, then $G$ contains a torsion free subgroup $H$ of finite index, and the {\it virtual cohomological dimension of $G$} is defined to be $\vcd(G) = \cd(H)$. A well-known theorem of Serre stablishes that $\vcd(G)$ is well defined and it does not depend on the choice of the finite index torsion free subgroup $H$ of $G$. 
For every virtually torsion free group $G$, we have the following inequalities (see \cite[Theorem 2]{BLN01})
\begin{equation}
    \vcd(G)\leq \cdfin(G) \leq \gdfin(G) \leq \max\{3, \cdfin(G)\}.
\end{equation}


Following  the general strategy of \cite{AMP14}, we will determine the proper geometric dimension $\gdfin(\Mod_g^n)$ by computing the algebraic counterpart $\cdfin(\Mod_g^n)$. We recall here the two main properties of $\cdfin(G)$ that we will need to obtain \cref{thm:main}.

Firstable, the following theorem is a consequence of 
\cite[Theorem 13.19]{Lu89}, as explained in \cite[Section 3]{AMP14}; see also \cite[Theorem 1]{BLN01}.

\begin{theorem}\label{thm:Luck} Let $G$ be  a group with $\cdfin(G)=d\geq 3$. Then there is a $d$-dimensional $\underbar{E}G$. Morevoer, if $G$ has a cocompact model for $\underbar{E}G$, then it also admits a cocompact $\underbar{E}G$ of dimension $d$.
\end{theorem}

Now, let $F$ be a finite subgroup of $G$. We denote by $Z_G(F)$, $N_G(F)$, and $W_G(F)=N_G(F)/F$  the centralizer, the normalizer, and the {\it Weyl group of} $F$ respectively. If there is no risk of confusion we will omit the parenthesis and subindices, i.e. we will use the notation $ZF$, $NF$, and $WF$. 
The \emph{length} $\lambda(F)$ of a finite group $F$ is the largest $i\geq 0$ for which  there is a sequence \[1=F_0<F_1 < \cdots < F_i=F.\]

The second useful result to obtain our \cref{thm:main} is the following theorem by Aramayona and Martínez Pérez that gives us a criterion to compute the proper cohomological dimension. It is a consequence from a result of Connolly and Kozniewski \cite[Theorem A]{CK86}. 

\begin{theorem}\label{thm:aramayona:martinezperez} Let $G$ be a virtually torsion free group such that for any $F\leq G$ finite subgroup, $\vcd(WF)+\lambda(F)\leq \vcd(G)$. Then $\cdfin(G)= \vcd(G)$. 
\end{theorem}

\begin{remark}  For $F$ is a finite subgroup of $G$ the normalizer $NF$ and the Weyl group $WF$ are commensurable, therefore $\vcd(NF)=\vcd(WF)$. In our analysis below,  we will obtain upper bounds of $\vcd(NF)$ in order to apply \cref{thm:aramayona:martinezperez}. 
    
\end{remark}




The following lemma will be useful in our computations below. For a given $g\in G$, we denote (respectively) by $Z(g)$, $N(g)$ and $W(g)$  the centralizer, the normalizer and the Weyl group of the cyclic subgroup $\langle g\rangle$ in $G$.

\begin{lemma}\label{lem:length:vcdWF}
    Let $G$ be a virtually torsion free group and let $F$ a finite subgroup. Then
    \begin{enumerate}
        \item $\lambda(F)\leq \log_2(|F|)$, and
        \item $\vcd(WF)\leq \min \{\vcd(Z(g))|g\in F\}$.
    \end{enumerate}
\end{lemma}
\begin{proof}
    Item (1) is \cite[Lemma 3.1]{HSST}.
    For item (2) note that $ZF$ is commensurable with $WF$, hence they have the same $\vcd$. The conclusion follows from the fact that $ZF=\cap_{g\in F}Z(g)$ and the monotonicity of the $\vcd$.
\end{proof}

\subsection{Mapping class groups}\label{mcg} 

For any $g\geq 0$ and $n\geq 0$, the group $\Mod_{g}^n$ is virtually torsion free and its virtual cohomological dimension was computed by J. L. Harer in \cite{Ha86}: 
\begin{equation}\label{vcd:mcg}
\vcd(\modgn) =
\begin{cases}
0& \text{if } g=0 \text{ and } n< 3\\
n-3 & \text{if } g=0 \text{ and } n\geq 3 \\
1 & \text{if } g=1 \text{ and } n= 0 \\
4g-5 & \text{if } g\geq 2 \text{ and } n=0\\
4g-4+n & \text{if } g\geq 1 \text{ and } n\geq 1.
\end{cases}
\end{equation}

Let $F$ be a finite subgroup of $\modgn$. By Nielsen realization (\cite[Theorem 7.2]{FM12},\cite{KER}), there is a hyperbolic structure for $S_g^n$  on which $F$ acts by orientation-preserving isometries,
when $\chi(S_g^n)<0$. 

By a slight abuse of notation, we denote by $S_g^n/F$ both the orbifold and the underlying topological surface. In what follows we will use the following notation: 
\begin{itemize}
    \item$g_F$ is the genus of $S_g^n/F$,
    \item $o_F$ is the number of elliptic (orbifold) points  of $S_g^n/F$,
    \item $n_F$ the number $o_F$ of elliptic points plus the number of orbits of punctures of this action
    \item$(g_F;p_1^F,\ldots,p_{o_F}^F)$ is often called the {\it signature of $F$}, where $p_1^F,\ldots,p_{o_F}^F$ are the orders of the elliptic points of $S_g^n/F$.
\end{itemize}

Notice that $n_F\leq \frac{n}{|F|}+o_F$.

The following lemma is due to Maher. Together with Harer's computation \cref{vcd:mcg}, it gives us  a way to compute the $\vcd$ of the Weyl group $WF$ of any finite subgroup $F$ of $\Mod_g^n$.

\begin{lemma}\label{lemma:Maher}
    \cite[Proposition 2.3]{Ma11} 
    Let $F\leq \Mod_g^n$ be a finite subgroup with $g_F$ and $n_F$ as before. Then $NF$ has finite index in $\Mod_{g_F}^{n_F}$. In particular, $$\vcd(WF) =\vcd (NF)= \vcd(\Mod_{g_F}^{n_F}).$$
\end{lemma}

\section{Genus 0}\label{sec:zero}

The main result of this section is \cref{thm:main:genus:zero} which proves the case $g=0$ of \cref{thm:main}. In the first half of this section we study the finite subgroups of $\Mod_0^n$ and their actions on $S_0^n$, our analysis has as a starting point Stukow's classification of maximal finite subgroups of $\Mod_0^n$ given in \cite{STUKOW}. In the second half of the section we apply the results in the first half to prove \cref{thm:main:genus:zero}. It is worth saying that we dealt with cases $5\leq n\leq 11$ one by one, and all the computations of $\vcd(NF)$ and $\lambda(F)$ are summarized in \cref{appendixA} by means of several tables.

\subsection{Finite subgroups of $\Mod_0^n$ and their realization}
We start by recalling Stukow's classification of maximal finite subgroups of $\Mod_0^n$.
\begin{theorem}\cite[Theorem 3]{STUKOW}\label{thm:classification:stukow}
    A finite subgroup $F$ of $\Mod_0^n$ is a maximal finite subgroup of $\Mod_0^n$ if and only if $F$ is isomorphic to one of the following:
    \begin{enumerate}
        \item a cyclic group $\Z/(n-1)$ if $n\neq 4$,
        \item the dihedral group $D_{2n}$ of order $2n$,
        \item the dihedral group $D_{2(n-2)}$ of order $2(n-2)$ if $n=5$ or $n\geq 7$,
        \item the alternating group $A_4$ if $n\equiv 4$ or $10\ (\mathrm{ mod}\ 12)$,
        \item the symmetric group $\mathfrak{S}_4$ if $n\equiv 0,2,6,8,12,14,18$ or $20\ (\mathrm{ mod }\ 24)$,
        \item the alternating group $A_5$ if $n\equiv 0,2,12,20,30,32,42$ or $50\ (\mathrm{ mod }\ 12)$.
    \end{enumerate}
\end{theorem}
Stukow also studied the conjugacy classes of finite subgroups of $\Mod_0^n$.
\begin{theorem}\cite[Theorem 4, Corollary 5]{STUKOW}\label{thm:maximal:finite:stukow}
    Let $F$ be a finite subgroup of $\Mod_0^n$. Then the set of conjugacy classes of subgroups of $\Mod_0^n$ isomorphic to $F$ has two elements if
    \begin{enumerate}
        \item $F\cong \Z/2$ with $n$ even.
        \item $F\cong D_{2m}$ with $2m|n$ or $2m|n-2$, 
    \end{enumerate}
    and one element otherwise. In particular, any two maximal finite subgroups of $\Mod_0^n$ are conjugate if and only if they are isomorphic.
\end{theorem}

Since, for a given finite subgroup $F$ of $\Mod_0^n$, the number $n_F$ only depends on the conjugacy class of $F$, we only have to work with a representative of such conjugacy class. Let us describe representatives of the conjugacy classes of maximal finite subgroups of $\Mod_0^n$ of types (1)-(3) following the numeration of \Cref{thm:maximal:finite:stukow}. For this let us denote by $\Nte$ and $\Sur$ the north and south poles of $S_0$ respectively. Also given any group isomorphic to $D_{2m}$ recall that there is an index 2 cyclic subgroup of order $m$, we call this subgroup the \emph{subgroup of rotations} and every element which do not belong to the subgroup of rotations is called a \emph{reflection} (all such elements have order 2).

\begin{enumerate}
    \item Up to a self homeomorphism of $S_0^n$, the action of $\Z/(n-1)$ can be realized as follows: we place one of the punctures on $\Nte$ and the remaining $n-1$ on the equator equidistantly. The generator of $\Z/(n-1)$ acts on $S_0^n$ by a rotation of angle $2\pi/(n-1)$ that fixes $\Nte$ and $\Sur$.

    \item  Up to a self homeomorphism of $S_0^n$, the action of $D_{2n}$ can be realized as follows: we place all the punctures along the equator equidistantly. As in the previous case we have a natural action of the subgroup of rotations $\Z/n$ on $S_0^n$, and a reflection can be realized as a half turn that interchanges $\Nte$ and $\Sur$ and leave the set of punctures invariant.

    \item The realization of $D_{2(n-1)}$ is completely analogous to the previous case, the only difference is that we place 2 punctures on $\Nte$ and $\Sur$ and the remaining $n-2$ on the equator.
\end{enumerate}

The numeration in the following proposition is designed to be compatible with the numeration in \cref{thm:classification:stukow}. We will use the above descriptions of the actions of maximal finite groups in the proof of the following statement.

\begin{proposition}\label{prop:nF}
    Let $F$ be a finite subgroup of $\Mod_0^n$. Then the following holds:
    \begin{enumerate}
        \item Let $F\cong \Z/m\leq \Z/(n-1)$. Then 
        \[n_F=2+\frac{n-1}{m}.\]
        
        \item[(2.1)] Let $F\cong \Z/m$ be  a subgroup of the rotation subgroup $\Z/n$ of $D_{2n}$. Then
        \[n_F=2+\frac{n}{m}.\]

        \item[(2.2)] Let $F\cong D_{2m}\leq D_{2n}$ with $m>1$.\\ 
       If $n$ odd, then 
        \[n_F=1+\frac{n+3m}{2m}.\]
        If $n$ even and $2m|n$, then
        \[n_F=\begin{cases}
            1+\frac{n+2m}{2m}\\
          1+\frac{n+4m}{2m}
        \end{cases}\]
        corresponding to the two conjugacy classes of subgroups isomorphic to $D_{2m}$.\\ If $n$ even and $2m\not |n$, then 
        \[n_F=1+\frac{n+3m}{2m}.\]

        \item[(3.1)] Let $F\cong \Z/m$ be a subgroup of the rotation subgroup $\Z/(n-2)$ of $D_{2(n-2)}$. Then
        \[n_F=2+\frac{n-2}{m}.\]

        \item[(3.2)] Let $F\cong D_{2m}\leq D_{2(n-2)}$ with $m>1$.\\ If $n$ odd, then 
        \[n_F=1+\frac{n-2+3m}{2m}.\]
        If $n$ even and $2m|(n-2)$, then
        \[n_F=\begin{cases}
            1+\frac{n-2+2m}{2m}\\
            1+\frac{n-2+4m}{2m}
        \end{cases}\]
        corresponding to the two conjugacy classes of subgroups isomorphic to $D_{2m}$.\\ If $n$ even and $2m\not |(n-2)$, then 
        \[n_F=1+\frac{n-2+3m}{2m}.\] 
    \end{enumerate}    
\end{proposition}

\begin{proof} We proceed by cases:
\begin{enumerate}
    \item[(1)] We can realize the $\Z/m$-action placing a puncture in $\Nte$ and the remaining $n-1$ equidistantly on the equator. Hence \{$\Nte$\} and \{$\Sur\}$ are both orbits of elliptic points of the $\Z/m$-action, while the orbits of punctures are are exactly $(n-1)/m$.


    \item[(2.1)] The $\Z/m$-action can be realized placing all punctures on the equator equidistantly. The conclusion is completely analogous to case (1). 

    \item[(2.2)] Recall that $D_{2m}$ is generated by the rotation subgroup $\Z/m\leq \Z/n$ and a half turn $\tau$ that interchanges $\Nte$ and $\Sur$. In this case we have the orbit $\{\Nte,\Sur\}$. On the other hand, note that the $D_{2n}$-action is free away from the poles and the equator. Hence any elliptic point other than $\Nte$ and $\Sur$ is on the equator. Let $X$ be the set of all punctures and all the elliptic points on the equator. Note that, for elements of $D_{2m}$, the identity element fix all points of $X$, all nontrivial rotations do not fix points in $X$, and each of the $m$ reflections fix two points in $X$, hence by the Burnside lemma, the number of orbits of $X$ is 
    $\frac{|X|+2m}{2m}$.
    Therefore, if $e$ is the number of elliptic points in $X$ we have
    \[n_F=1+\frac{n+e+2m}{2m}.\]

    Assume $n$ is odd, then the axis points of $\tau$ consists exactly of one puncture $p$ and its antipode $-p$ which is an elliptic point. Therefore there are $m$ elliptic points, that is $e=m$.

    Assume $n$ is even. Then the axis of reflections in $D_{2m}$ are obtained by $\Z/m$-rotating the axis of $\tau$ and taking the bisectors of any two consecutive rotations of the axis of $\tau$.

     If $2m$ divides $n$, $n/m$ is even. Let $p$ be a fixed puncture and let $t$ be a fixed generator of $\Z/m$. Then between $p$ and $tp$ there are $n/m-1$ punctures, that is, an odd number of punctures. If $\tau$ fixes two punctures, then each bisector also fixes two punctures. Therefore there are no elliptic points and $e=0$. If $\tau$ fixes two elliptic points, then each bisector also fixes two elliptic points. Therefore $e=2m$.
    
    If $2m$ does not divide $n$, then $n/m$ is odd and $m$ is even. Let $p$ be a fixed puncture and let $t$ be a fixed generator of $\Z/m$. Then between $p$ and $tp$ there are $n/m-1$ punctures, that is, an even number of punctures. Hence, whenever $\tau$ fixes two punctures then each bisector fixes two elliptic points, and whenever $\tau$ fixes two elliptic points then each bisector fixes two punctures. In both scenarios we conclude $e=m$.

\end{enumerate}

Cases (3.1) and (3.2) are analogous to cases (2.1) and (2.2) respectively and are left as an exercise to the reader. 
\end{proof}

\subsection{Proof of the main theorem for genus 0} 
We compute in this subsection, \cref{thm:main:genus:zero} below, the geometric dimension for $\Mod_0^n$.

\begin{lemma}\label{lemma:vcdNF:sphere}
    Let $F$ be a finite subgroup of $\Mod_0^n$, then $\vcd(NF)=\max\{n_F-3,0\}$.
\end{lemma}
\begin{proof}
    By Nielsen realization $F$ acts on $S_0^n$ by orientation preserving homeomorphisms.
    As a straightforward application of the Riemann-Hurwitz theorem, we have that $S_0^n/F$ is homeomorphic to a sphere, that is $g_F=0$. Now the result is a consequence of \Cref{lemma:Maher} and Harer's computation of the virtual cohomological dimension for mapping class groups given in \cref{vcd:mcg}.
\end{proof}

\begin{proposition}\label{prop:Weyl:cyclic}
    Let $n\geq 3$ and let $g\in \Mod_0^n$ be an element of finite order $r\geq 2$. Then $r\leq n$ and
    \[\vcd(W(g))\leq \frac{n}{r}-1\leq \frac{n}{2}-1\]
\end{proposition}
\begin{proof}
By \cite[Theorem 4, Corollary 5]{STUKOW} the element $g$ is contained in a unique (up to isotopy) maximal finite cyclic subgroup of order $n$, $n-1$ or $n-2$, which by \cref{thm:maximal:finite:stukow} and \cref{thm:classification:stukow}, is the rotation subgroup of  one of the maximal subgroups $\Z/(n-1)$, $D_{2n}$ or $D_{2(n-2)}$. Hence, up to conjugation, $g$ can be realized as a rotation that fixes $\Nte$ and $\Sur$ and the punctures are on the equator and possible in one or both poles. Therefore $n_F\leq 2 + n/r$, and the first inequality in our statement follows from \cref{lemma:vcdNF:sphere}. The second inequality is clear.
\end{proof}

\begin{proposition}\label{prop:vcdpluslegth:cyclic:case}
    Let $n\geq 11$ and $g\in \Mod_0^n$ be an element of order $r\geq 2$. Then
    \[\vcd(W(g))+\lambda(\langle g \rangle)\leq \vcd(\Mod_0^n).\]
\end{proposition}
\begin{proof} By \cref{prop:Weyl:cyclic} and \cref{lem:length:vcdWF} we get
    \begin{align*}
        \vcd(W(g))+\lambda(\langle g \rangle)&\leq \frac{n}{r}-1+\log_2(r)\\
        &\leq \frac{n}{r}-1+\log_2(n)\\
        &\leq \frac{n}{2}+\log_2(n)
    \end{align*}

    Since both $\frac{n}{2}+\log_2(n)$ and $\vcd(\Mod_0^n)=n-3$ are both increasing functions of $n$, it is easy to verify that
    \[\frac{n}{2}+\log_2(n)\leq n-3\]
    for $n\geq 11$. 
\end{proof}

\begin{theorem}\label{thm:vcdWF:lambdaF:inequality}
    Let $n= 5$ or $n\geq 7$, and let $F$ be a finite subgroup of $\Mod_0^n$. Then 
    \[\vcd(WF)+\lambda(F)\leq \vcd(\Mod_0^n)=n-3.\]
\end{theorem}
\begin{proof}
    For $n=5$ or $7\leq n \leq 13$, the conclusion follows from the tables in \cref{appendixA}, where $\vcd(WF)$ is computed  using \cref{prop:nF} and \cref{lemma:vcdNF:sphere}. The remaining of the proof deals with the case $n\geq 14$.
    By \cref{thm:maximal:finite:stukow}, up to conjugation, $F$ is a subgroup of one of six possibilities:

    \begin{enumerate}
        \item $F\leq \Z/(n-1)$. The conclusion follows directly from \cref{prop:vcdpluslegth:cyclic:case} for $n\geq 11$.

        \item $F\leq D_{2n}$. If $F$ is cyclic, the result follows from \cref{prop:vcdpluslegth:cyclic:case} for $n\geq 11$. Assume that $F\cong D_{2m}$ with $m|n$. Then by \cref{lem:length:vcdWF} and \cref{prop:Weyl:cyclic} we get
        \[ \vcd(WF)\leq \vcd(\Z/m)\leq \frac{n}{m}-1\quad\text{and}\quad \lambda(F)\leq \log_2(m)+1. \]
        Thus \[\vcd(WF)+\lambda(F) \leq \frac{n}{m} + \log_2(m)\leq \frac{n}{2}+\log_2(n).\]
        Now note that $\frac{n}{2}+\log_2(n)\leq n-3$ for $n\geq 14$.

        \item $F\leq D_{2(n-2)}$. This case is analogous to the previous one.

        \item $F\leq A_4$. We have $\lambda(F)\leq \lambda(A_4)=3$. By \cref{lem:length:vcdWF} and \cref{prop:Weyl:cyclic} we get
        \[\vcd(WF)+\lambda(F)\leq \frac{n}{2}+3.\]
        Note that $\frac{n}{2}+3\leq n-3$ for $n\geq 10$.
        \item $F\leq \mathfrak{S}_4$. Analogous to case (4).
        \item $F\leq A_5$. Analogous to case (4).
    \end{enumerate}
    
\end{proof}

\begin{theorem}\label{thm:main:genus:zero}
    Let $n\geq 0$. Then 
    \[\gdfin(\Mod_0^n)= \cdfin(\Mod_0^n)= \vcd(\Mod_0^n)=\max\{n-3,0\}.\]
\end{theorem}
\begin{proof} We divide the proof in cases.

    Let $n=0,1,2,3$. In these cases $\Mod_0^n$ is a finite group and the conclusion follows. 
    
    Let $n=4$. Since $\vcd(\Mod_0^4)=1$, we know that $\Mod_0^4$ is virtually finitely generated free, hence by a well-known theorem of Stallings, $\Mod_0^4$ acts properly on a tree and this action provides a model for $\underline E \Mod_0^4$ of dimension 1, therefore $\gdfin(\Mod_0^4)=\cdfin(\Mod_0^4)=1$.
   
    Let $n=6$. In this case, by Birman-Hilden theory we have a central extension 
    \[1\to \Z/2\to \Mod_2 \to \Mod_0^6\to 1,\]
where the kernel is generated by the hyperelliptic involution of $S_2$.    By \cite[Proposition 4.3]{Ji14} there is a model $X$ for $\underline E \Mod_2$ of dimension 3. Therefore $X^{\Z/2}$ is a model for $\underline E \Mod_0^6$ of dimension at most 3. We conclude that $\cdfin(\Mod_0^6)=3$.
    
    Let $n=5$ or $n\geq 7$. From \cref{thm:vcdWF:lambdaF:inequality} and \cref{thm:aramayona:martinezperez} we obtain the proper cohomological dimension. By the Eilenberg-Ganea theorem, we have that $\gdfin(\Mod_0^n)=\cdfin(\Mod_0^n)$, except possibly when $n=5$. To rule out that possibility, we consider the following central extension that arises from  Birman-Hilden theory
    $$1\rightarrow \mathbb{Z}/2\rightarrow \Mod_1^2\rightarrow \Mod_0^5\rightarrow 1,$$
    where the kernel is generated by the hyperelliptic involution of $S_1^2$. If follows from \cite[Lemma 5.8]{Lu05}  that $\gdfin(\Mod_0^5)=\gdfin(\Mod_1^2)$. Hence, $\gdfin(\Mod_0^5)=2$ follows from our computation $\gdfin(\Mod_1^2)=\cdfin(\Mod_1^2)=2$ in the proof of \cref{thm:main:genus:one}. 
\end{proof}

\section{Genus 1}\label{sec:one}

In this section we prove  \cref{thm:main:genus:one} which gives the case $g=1$ of \cref{thm:main}. First we consider finite subgroups of $\Mod_1^n$ and their actions on $S_1^n$.

\begin{theorem}\label{thm:classification:finite:torus}
 Let $n\geq 0$, and let $F$ be a finite subgroup of orientation-preserving  diffeomorphisms of $S_1^n$, then $F$ is isomorphic to $(\Z/s\times \Z/t)\rtimes \Z/m$ where $s$ and $t$ are positive integers such that $st|n$ and $m=1,2,3,4$ or $6$. Moreover, the quotient $S_1/F$ only depends on $m$, and we have the following descriptions of these orbifold quotients:
 \begin{itemize}
     \item for $m=1$, $S_1/F$ is a torus with no elliptic points,
     \item for $m=2$, $S_1/F$ is a sphere with four elliptic points of orders $(2,2,2,2)$,
     \item for $m=3$, $S_1/F$ is a sphere with three elliptic points of orders $(3,3,3)$,
     \item for $m=4$, $S_1/F$ is a sphere with three elliptic points of orders $(2,4,4)$, and
     \item for $m=6$, $S_1/F$ is a sphere with three elliptic points of orders $(2,3,6)$.
 \end{itemize}
\end{theorem}
\begin{proof}
    Let $F$ be as in the statement, by capping the punctures, the action of $F$ on $S_1^n$ leads to an action of $F$ on the torus $S_1$. By the uniformization theorem, there is a metric of constant curvature 0 on $S_1$ such that the action of $F$ is by isometries. Next we can lift this action to the universal covering, that is, there is a group $\widetilde{F}$ acting by orientation preserving euclidean isometries on  $\mathbb R^2$  that is a central extension of $F$ by a rank $2$ group of translations, and such that $\mathbb{R}^2/\widetilde{F}$ is diffeomorphic to $S_1/F$. Hence $\widetilde{F}$ is a wallpaper group without reflections nor glide reflections, and therefore $\widetilde{F}$ is isomorphic to $\Z^2\rtimes \Z/m$ for $m=1,2,3,4,$ or $6$. This implies that there are positive integers $s$ and $t$ such that $F\cong (\Z^2\rtimes \Z/m)/(s\Z \oplus t\Z)\cong (\Z/s\times \Z/t)\rtimes \Z/m$. Furthermore, note that $\Z/s\times \Z/t$ acts freely on $S_1$ and preserves the set of punctures in $S_1^n$, which is a set with $n$ elements, as a conclusion $st|n$. 

    For the \textit{moreover part} of the statement, observe that $S_1/F$ is diffeomorphic to $\mathbb  R^2/(\Z^2\rtimes \Z/m)$, and the latter quotients are well known, see for instance \cite[Theorem 13.3.6]{ThurstonBook}.
\end{proof}

\begin{proposition}\label{prop:inequality:genus1}
    Let $n\geq 2$. Then for every finite subgroup $F$ of $\Mod_1^n$ we have
    \[\vcd(WF)+\lambda(F)\leq \vcd(\Mod_1^n)=n.\]
\end{proposition}
\begin{proof} The conclusion is clear when $F$ is the trivial subgroup.
    Let $F$ be a nontrivial finite subgroup of $\Mod_1^n$, then by Nielsen realization theorem we can realize $F$ as a finite group of orientation-preserving  diffeomorphisms of $S_1^n$. We proceed by cases following the notation and conclusions in \Cref{thm:classification:finite:torus}.

    \begin{itemize}
        \item $F\cong \Z/s\times \Z/t$ with $st|n$. In this case we have $n_F=n/st$ and $S_1^n/F$ is diffeomorphic to $S_1$ with $n/st$ punctures. Hence by \Cref{lem:length:vcdWF}(1) we get
        \[\vcd(WF)+\lambda(F)\leq \vcd(\Mod_1^{n_F})+\lambda(F)\leq \frac{n}{st}+\log_2(n)\leq \frac{n}{2}+\log_2(n).\]
        Now notice that $n/2+\log_2(n)\leq n$ for all $n\geq 0$.

        \item $F\cong ( \Z/s\times \Z/t)\rtimes \Z/2$. In this case we have $S_1/F$ is diffeomorphic to a sphere with 4 elliptic points. Hence we conclude that $n_F\leq 4+n/st\leq 4+n/2$, and $\vcd(WF)=n_F-3$. Therefore
        \[\vcd(WF)+\lambda(F)\leq \frac{n}{2}+1+\log_2(2n)\leq \frac{n}{2}+2+\log_2(n). \]
    Now notice that $n/2+2+\log_2(n)\leq n$ for all $n\geq 5$.

    \item  $F\cong (\Z/s\times \Z/t)\rtimes \Z/m$ with $m=3,4,6$. In this case we have $S_1/F$ is diffeomorphic to a sphere with 3 elliptic points. Hence we conclude that $n_F\leq 4+n/st\leq 3+n/2$, and $\vcd(WF)=n_F-3$. Therefore
        \[\vcd(WF)+\lambda(F)\leq \frac{n}{2}+\log_2(6n)\leq \frac{n}{2}+2+\log_2(n). \]
    Now notice that $n/2+2+\log_2(n)\leq n$ for all $n\geq 5$.
    \end{itemize}

    To finish the proof we only have to verify the statement for $n=2,3,4$. From \Cref{thm:classification:finite:torus} we can describe explicitly the finite subgroups of $\Mod_1^n$ and their quotients, and the analysis can be done case by case; the details are left to the reader.
\end{proof}

\begin{theorem}\label{thm:main:genus:one}
    Let $n\geq 0$. Then 
    \[\gdfin(\Mod_1^n))=\cdfin(\Mod_1^n)= \vcd(\Mod_1^n)=n.\]
\end{theorem}
\begin{proof} For $n=1$, the result follows since $\Mod_1^1\cong\text{SL}(2,\mathbb{Z})$ is virtually free.
For $n\geq 2$,  \cref{prop:inequality:genus1} and \cref{thm:aramayona:martinezperez} imply $\cdfin(\Mod_1^n)= \vcd(\Mod_1^n)$.

By the Eilenberg-Ganea theorem, we have that $\gdfin(\Mod_1^n)=\cdfin(\Mod_1^n)=n$, except possibly for $n=2$. Let us rule out this possibility. We have the following short exact sequence
\[1\to B_2(S_1)/Z \to \Mod_1^2 \to \Mod_1  \to 1,\] 
where $B_2(S_1)$ is the braid group over the torus on 2 strands, and $Z$ is its center. On the other hand $B_2(S_1)/Z\cong \Z/2*\Z/2*\Z/2 $,
see for instance \cite[Proposition 4.4(i)]{bellingeri}. Note that  $\Z/2*\Z/2*\Z/2$ is virtually finitely generated free, hence every finite extension $G$ of this group admits a proper action  on a tree, which provides a 1-dimensional model for  $\underline E G$. The group $\Mod_1$ also admits a 1-dimensional model for $\underline E \Mod_1$, hence by \cite[Theorem 5.16]{Lu05}, there is a 2-dimensional model for $\underline E \Mod_1^2$. This establishes that $\gdfin(\Mod_1^2)=\cdfin(\Mod_1^2)=2$. 
\end{proof}

 
\section{Genus 2}\label{sec:two}

 In order to compute the proper geometric dimension of $\Mod_2^n$, with $n\geq 1$, we  use Broughton's complete classification, up to topological equivalence, of finite group actions on a genus $2$ surface \cite[Theorem 4.1 $\&$ Table 4]{FiniteGenus2}, see \cref{appendixB}.  Notice that there are only finitely many conjugacy classes of finite groups that act on $S_2$ by homeomorphism. Hence by Nielsen realization theorem, given $n\geq 0$, any finite subgroup $F$ of $\Mod_2^n$ can be realized, up to conjugation, by one of these finite possibilities. This makes the genus 2 case different in nature to the cases of genus 0 and 1 where the isomorphism types of finite subgroups of the corresponding mapping class groups depend strongly on $n$.


\begin{theorem} \label{thm:main:genus:2}
    Let $n\geq 1$. Then for every non-trivial finite subgroup $F$ of $\Mod_2^n$ we have
    \[\vcd(WF)+\lambda(F)\leq \vcd(\Mod_2^n)=n+4.\]
    In particular $\cdfin(\Mod_2^n)=\gdfin(\Mod_2^n)=\vcd(\Mod_2^n)=n+4$.
\end{theorem}
\begin{proof} Since $\cdfin(\Mod_2^n)\geq \vcd(\Mod_2^n)\geq 5$, we have $\cdfin(\Mod_2^n)=\gdfin(\Mod_2^n)$ for every $n\geq 1$.

Let $F$ be a non-trivial finite subgroup of $\Mod_2^n$. By Nielsen realization $F$ acts on $S_2^n$ by orientation preserving homeomorphisms. Recall from \cref{mcg} that $n_F$ is bounded above by $\frac{n}{|F|}+o_F$. 

According to Broughton's  classification the quotient $S^n_2/F$ can only have genus $g_F=0$ or $g_F=1$. When $g_F=0$, we have that $\vcd(WF)=\vcd(\Mod_{0}^{n_F})=n_F-3$
and when $g_F=1$, then $\vcd(WF)=\vcd(\Mod_{1}^{n_F})=n_F$. In \cref{Genus2} we recall the classification from \cite[Table 4]{FiniteGenus2}, and we give explicit upper bounds for $\lambda(F)$, $n_F$ and $\vcd(WF)$. From this we can see that the inequality  $\vcd(WF)+\lambda(F)\leq \vcd(\Mod_2^n)=n+4$ holds for almost all finite subgroups $F$ of $\Mod_2^n$ and  $n\geq 1$. The only exception is the case of $F=\text{GL}_2(4)$ and $n=1$, which does not occur since $F=\text{GL}_2(4)$ cannot be realized as subgroup of $\Mod_2^1$; indeed since the only puncture will be a fixed point of the action, but as we can read in the signature of this action in \cref{Genus2} there are no fixed points.
\end{proof}

\begin{remark}
From \cref{Genus2} we can see that there are several finite subgroups $F$ of $\Mod_2$ such that 
$\vcd(WF)+\lambda(F)\not\leq \vcd(\Mod_2)=3$. For instance, $\vcd(WF)+\lambda(F)=4$ when $F=D_{2(6)}$. Hence, we cannot use this strategy to obtain $\gdfin(\Mod_2)$.
\end{remark}

\section{Genus at least 3}\label{sec:atleast3}

In this section we promote the results in \cite{AMP14} for $\Mod_g^0$ with $g\geq 3$, to $\Mod_g^n$ for $g\geq 3$ and $n\geq 1$, using a simple argument.

\begin{proposition}\cite[Proposition 4.4]{AMP14}\label{AMP:g:atleast:three}
 For any $g\geq 3$ and any finite subgroup $F$ of $\mathrm{Mod}_g^0$ we have
 \[\vcd(WF)+\lambda(F)\leq \vcd(\mathrm{Mod}_g^0).\]
\end{proposition}

\begin{theorem}\label{thm:main:genus:atleast3}
    For any $g\geq 3$, $n\geq 1$ and any finite subgroup $F$ of $\mathrm{Mod}_g^n$ we have
 \[\vcd(WF)+\lambda(F)\leq \vcd(\mathrm{Mod}_g^n).\]
 In particular $\gdfin(\mathrm{Mod}_g^n)=\cdfin(\mathrm{Mod}_g^n)=\vcd(\mathrm{Mod}_g^n)$.
\end{theorem}
\begin{proof}
    Let us consider the Birman short exact sequence \cite[Theorem 4.3]{BIRMAN}
    \[1\to B_n(S_g) \to \modgn \xrightarrow{\varphi} \mathrm{Mod}_g^0\to 1\]
    where $B_n(S_g)$ is the full braid group over $S_g$ on $n$ strings.

    Let $F$ be any finite subgroup of $\modgn$. Then by restriction of the above sequence we obtain the following short exact sequence
    \[1\to B_n(S_g)\cap NF \to NF \xrightarrow{\varphi} \varphi(NF)\to 1\]
    and note that $\varphi(NF)\leq N\varphi(F)$. Therefore 
    \begin{align*}
        \vcd(NF)+\lambda(F) &\leq \vcd(\varphi(NF))+\vcd(B_n(S_g)\cap NF)+\lambda(F),\text{ by subaditivity of $\vcd$}\\
        &\leq \vcd(N(\varphi(F)))+\vcd(B_n(S_g))+\lambda(\varphi(F)),\text{ by monotonicity of $\vcd$}\\
        &\leq \vcd(\mathrm{Mod}_g^0)+\vcd(B_n(S_g)), \text{ by \cref{AMP:g:atleast:three}}\\
        &\leq 4g-5+n+1=4g+n-4=\vcd(\modgn), \text{by \cite[Theorem 1.2]{MR3869010}.}
    \end{align*}
    Now the result follows. The \textit{in particular} part follows directly from \cref{thm:aramayona:martinezperez}.
\end{proof}

\clearpage

\input{appendix}

\bibliographystyle{alpha} 
\bibliography{mybib}
\end{document}

%% file: Preamble.tex
\usepackage{graphicx}
\usepackage{pinlabel} 
\usepackage{amsmath}
\usepackage{amsfonts}
\usepackage{amssymb}
\usepackage{mathtools}
\usepackage[all,cmtip]{xy}
\usepackage{amsthm}
\usepackage{tikz-cd}
\usepackage{comment}
\usepackage{enumerate}
\usepackage{url}
\usepackage{epsfig}
\usepackage{hyperref}
\usepackage[utf8]{inputenc}
\usepackage[capitalize]{cleveref}
\usepackage{geometry}

\usepackage{subcaption}
\usepackage{booktabs}

\usepackage{mathrsfs}
\usepackage{xcolor}

\newtheorem{theorem}{Theorem}[section]
\newtheorem{proposition}[theorem]{Proposition}
\newtheorem{corollary}[theorem]{Corollary}
\newtheorem{lemma}[theorem]{Lemma}

\newcommand{\mycomment}[1]{}

  \theoremstyle{definition}

\newtheorem{remark}{Remark}

\newcommand{\Z}{\mathbb{Z}}

\newcommand{\T}{\mathscr T}

\newcommand{\nbeq}{\begin{equation}}
\newcommand{\neeq}{\end{equation}}
\newcommand{\beq}{\begin{equation*}}
\newcommand{\eeq}{\end{equation*}}

\DeclareMathOperator{\cd}{cd}

\DeclareMathOperator{\vcd}{vcd}
\DeclareMathOperator{\cdfin}{\underline{cd}}

\DeclareMathOperator{\gdfin}{\underline{gd}}

\DeclareMathOperator{\Mod}{Mod}
\DeclareMathOperator{\PMod}{PMod}

\newcommand{\modgn}{\Mod_{g}^{n}}

\newcommand{\Tgn}{\T_{g}^{n}}

\DeclareMathOperator{\Nte}{\text{\sc N}}
\DeclareMathOperator{\Sur}{\text{\sc S}}



%% file: appendix.tex
\appendix
\section{$\lambda(F)$ and $\vcd(WF)$ for finite subgroups of $\Mod_0^n$ for $5\leq n\leq 13$}\label{appendixA}

 In this appendix we describe $n_F$, $\vcd(WF)$ and $\lambda(F)$ for all the conjugacy classes of  finite subgroups $F$ of $\Mod_0^n$ when $5\leq n\leq 13$.\\

\noindent{\bf Polyhedral subgroups of $\Mod_0^n$.} We keep the numeration from \cref{thm:classification:stukow} to analyze the polyhedral subgroups of $\Mod_0^n$ for $n=6,8,10,12$: 
\begin{itemize}
 \item[(3)] For $n=10$, up to conjugacy, there is a maximal finite subgroup $F\cong A_4$ of $\Mod_0^{n}$, which can be realized as the symmetry group of a tetrahedron. There are  three orbits of points with nontrivial stabiliser in $F$: the centers of faces, the centers of edges and the vertices with orbits of length $4$, $6$ and $4$, respectively. Notice that for $n=10$ there must be two orbits of punctures (one of length $4$ and another of length $6$) and one orbit of elliptics, hence $n_F=3$.
 \item[(4)] For $n=6,8,12$, up to conjugacy, there is a maximal finite subgroup $F\cong \mathfrak{S}_4$ of $\Mod_0^{n}$, which can be realized as the symmetry group of a cube (octahedron). There are  three orbits of points with nontrivial stabiliser in $F$: the centers of faces, the centers of edges and the vertices with stabilisers of order $4$, $2$ and $3$, respectively. Notice that for $n=6$ we must place the punctures in the centers of the faces (unique orbit of length $6$), for $n=8$ we must place the punctures in the vertices of the cube (unique orbit of length $8$) and for $n=12$ the punctures can only be placed in the centers of the edges (unique orbit of length $12$). Hence, in $S_0^n/F$ there is only one orbit of punctures and two orbits of elliptics and $n_F=3$. 

 Now let $F\cong A_4\leq \mathfrak{S}_4$. It can be realized as subgroup of the symmetry group of a tetrahedron embedded into a cube.  For $n=6$,  there are two orbits of vertices, one orbit of centers of faces (where we have placed the punctures)  and the centers of edges are no longer elliptics, hence $n_F=3$. For $n=8$,  there are two orbits of vertices (where we have placed the punctures), one orbit of centers of faces and the centers of edges are no longer elliptics, hence $n_F=3$. Finally, for $n=10$ there are two orbits of vertices, one orbit of centers of faces and one orbit of the centers of edges (where we have placed the punctures), , hence $n_F=4$.
 
 \item[(5)] For $n=12$, up to conjugacy, there is a maximal finite subgroup $F\cong A_5$ of $\Mod_0^{n}$, which can be realized as the symmetry group of a dodecahedron (icosahedron). There are three orbits of points with nontrivial stabiliser in $F$: the centers of faces, the centers of edges and the vertices. For $n=12$ we must place the punctures in the centers of the faces (unique orbit of length $12$). Hence, in $S_0^n/F$ there is only one orbit of punctures and two orbits of elliptics and $n_F=3$. 
\end{itemize}

We summarize the conclusions in \cref{Poly}. Notice that $\vcd(WF)+\lambda(F)\leq \vcd(\Mod_0^n)=n-3$ for $n=8,10,12$. However, $\vcd(W\mathfrak{S}_4)+\lambda(\mathfrak{S}_4)=4\not\leq \vcd(\Mod_0^6)=3$.

{\footnotesize
\begin{table}[ht]
    \centering
    {\caption{\small Polyhedral subgroups of $\Mod_0^n$, for $n=6,8,10,12$.}\label{Poly}} 
    \begin{tabular}{|c|c|c|c|c|}
        \hline  &&&&\\[-0.5em]
        $n$& $F$ & \textbf{${n_F}$} & \textbf{$\vcd(WF)$} & \textbf{$\lambda(F)$}\\   &&&&\\[-0.5em]
        \hline 
       $6$& $\mathfrak{S}_4$ & $3$ & $0$ & $4$ \\ 
        &$A_4$ &  $3$ & $0$& $3$ \\ 
         \hline 
      $8$& $\mathfrak{S}_4$ & $3$ & $0$ & $4$ \\ 
        &$A_4$ &  $3$ & $0$& $3$ \\ 
        \hline 
      $10$&  $A_4$ &  $3$ & $0$& $3$ \\ 
        \hline 
       $12$&   $\mathfrak{S}_4$ & $3$ & $0$ & $4$ \\ 
        & $A_4$ &  $4$ & $1$& $3$ \\ 
        & $A_5$ &  $3$ & $0$& $4$ \\ 
        \hline
    \end{tabular}

\end{table}  }\medskip

\clearpage
\noindent{\bf Cyclic and dihedral subgroups of $\Mod_0^n$.} The following tables for $5\leq n\leq 13$ come directly form \cref{prop:nF}.  Note that in all these cases $\vcd(WF)+\lambda(F)\leq \vcd(\Mod_0^n)=n-3$.

{\footnotesize
\begin{table}[ht]
    \centering
    \caption{\small Cyclic and dihedral subgroups of $\Mod_0^5$.}
    \begin{tabular}{|c|c|c|c|c|}
        \hline &&&&\\[-0.5em]
        \cref{prop:nF} type & $F$ & \textbf{$n_F$} & \textbf{$\vcd(WF)$} & \textbf{$\lambda(F)$}\\ &&&&\\[-0.5em]
        \hline
        (1) & $\mathbb{Z}/4$ & $3$ & $0$ & $2$ \\ 
        & $\mathbb{Z}/2$ & $4$ & $1$ & $1$ \\ 
        \hline
        (2.1) & $\mathbb{Z}/5$ & $3$ & $0$ & $1$ \\ 
        \hline
        (2.2) & $D_{2(5)}$ & $3$ & $0$ & $2$ \\ 
        \hline
        (3.1) & $\mathbb{Z}/3$ & $3$ & $0$ & $1$ \\ 
        \hline
        (3.2) & $D_{2(3)}$ & $3$ & $0$ & $2$ \\ 
        \hline
    \end{tabular}
\end{table}
}

{\footnotesize
\begin{table}[ht]
    \centering
    \caption{\small Cyclic and dihedral subgroups of $\Mod_0^6$.} 
    \begin{tabular}{|c|c|c|c|c|}
        \hline &&&&\\[-0.5em]
        \cref{prop:nF} type & $F$ & \textbf{$n_F$} & \textbf{$\vcd(WF)$} & \textbf{$\lambda(F)$}\\ &&&&\\[-0.5em]
        \hline
        (1) & $\mathbb{Z}/5$ & $3$ & $0$ & $1$ \\ 
        \hline
        (2.1) 
        & $\mathbb{Z}/6$ & $3$ & $0$ & $2$ \\
        & $\mathbb{Z}/3$ & $4$ & $1$ & $1$ \\
        & $\mathbb{Z}/2$ & $5$ & $2$ & $1$ \\
        \hline
        (2.2) & $D_{2(6)}$ & $3$ & $0$ & $3$ \\
        & $D_{2(3)}$ & $3$ & $0$ & $2$ \\ 
        &             & $4$ & $1$ & $2$ \\ 
        & $D_{2(2)}$ & $4$ & $1$ & $2$ \\ 
        \hline
        (3.1) & $\mathbb{Z}/4$ & $3$ & $0$ & $2$ \\
        & $\mathbb{Z}/2$   & $4$ & $1$ & $1$ \\
        \hline
        (3.2) & $D_{2(4)}$ & $3$ & $0$ & $3$ \\
        & $D_{2(2)}$ & $3$ & $0$ & $2$ \\
        &             & $4$ & $1$ & $2$ \\
        \hline
    \end{tabular}
\end{table}
}

{\footnotesize
\begin{table}[ht]
    \centering
    \caption{\small Cyclic and dihedral subgroups of $\Mod_0^7$.} 
    \begin{tabular}{|c|c|c|c|c|}
        \hline &&&&\\[-0.5em]
        \cref{prop:nF} type & $F$ & \textbf{$n_F$} & \textbf{$\vcd(WF)$} & \textbf{$\lambda(F)$}\\ &&&&\\[-0.5em]
        \hline
        (1) & $\mathbb{Z}/6$ & $3$ & $0$ & $2$ \\ 
        & $\mathbb{Z}/3$ & $4$ & $1$ & $1$ \\ 
        & $\mathbb{Z}/2$ & $5$ & $2$ & $1$ \\ 
        \hline
        (2.1) & $\mathbb{Z}/7$ & $3$ & $0$ & $1$ \\ 
        \hline
        (2.2) & $D_{2(7)}$ & $3$ & $0$ & $2$ \\ 
        \hline
        (3.1) & $\mathbb{Z}/5$ & $3$ & $0$ & $1$ \\ 
        \hline
        (3.2) & $D_{2(5)}$ & $3$ & $0$ & $2$ \\ 
        \hline
    \end{tabular}
\end{table}
}

{\footnotesize
\begin{table}[ht]
    \centering
    \caption{\small Cyclic and dihedral subgroups of $\Mod_0^8$.} 
    \begin{tabular}{|c|c|c|c|c|}
        \hline &&&&\\[-0.5em]
        \cref{prop:nF} type & $F$ & \textbf{$n_F$} & \textbf{$\vcd(WF)$} & \textbf{$\lambda(F)$}\\ &&&&\\[-0.5em]
        \hline
        (1) & $\mathbb{Z}/7$ & $3$ & $0$ & $1$ \\ 
        \hline
        (2.1) & $\mathbb{Z}/8$ & $3$ & $0$ & $3$ \\
        & $\mathbb{Z}/4$ & $4$ & $1$ & $2$ \\
        & $\mathbb{Z}/2$ & $6$ & $3$ & $1$ \\
        \hline
        (2.3) & $D_{2(8)}$ & $3$ & $0$ & $4$ \\
        & $D_{2(4)}$ & $3$ & $0$ & $3$ \\ 
        &            & $4$ & $1$ & $3$ \\ 
        & $D_{2(2)}$ & $4$ & $1$ & $2$ \\ 
        &            & $5$ & $2$ & $2$ \\ 
        \hline
        (3.1) & $\mathbb{Z}/6$ & $3$ & $0$ & $2$ \\
        & $\mathbb{Z}/3$ & $4$ & $1$ & $1$ \\
        & $\mathbb{Z}/2$ & $5$ & $2$ & $1$ \\
        \hline
        (3.2) & $D_{2(6)}$ & $3$ & $0$ & $3$ \\
        & $D_{2(3)}$ & $3$ & $0$ & $2$ \\
        &            & $4$ & $1$ & $2$ \\
        & $D_{2(2)}$ & $4$ & $1$ & $2$ \\
        \hline
    \end{tabular}
\end{table}
}

{\footnotesize
\begin{table}[ht]
    \centering
    \caption{\small Cyclic and dihedral subgroups of $\Mod_0^9$.} 
    \begin{tabular}{|c|c|c|c|c|}
        \hline &&&&\\[-0.5em]
        \cref{prop:nF} type & $F$ & \textbf{$n_F$} & \textbf{$\vcd(WF)$} & \textbf{$\lambda(F)$} \\ &&&&\\[-0.5em]
        \hline
        (1) & $\mathbb{Z}/8$ & $3$ & $0$ & $3$ \\
        & $\mathbb{Z}/4$ & $4$ & $1$ & $2$ \\
        & $\mathbb{Z}/2$ & $6$ & $3$ & $1$ \\
        \hline
        (2.1) & $\mathbb{Z}/9$ & $3$ & $0$ & $2$ \\
        & $\mathbb{Z}/3$ & $5$ & $2$ & $1$ \\
        \hline
        (2.2) & $D_{2(9)}$ & $3$ & $0$ & $3$ \\
        & $D_{2(3)}$ & $4$ & $1$ & $2$ \\
        \hline
        (3.1) & $\mathbb{Z}/7$ & $3$ & $0$ & $1$ \\
        \hline
        (3.2) & $D_{2(7)}$ & $3$ & $0$ & $2$ \\
        \hline
    \end{tabular}
\end{table}
}

{\footnotesize
\begin{table}[ht]
    \centering
    \caption{\small Cyclic and dihedral subgroups of $\Mod_0^{10}$.}
    \begin{tabular}{|c|c|c|c|c|}
        \hline &&&& \\[-0.5em]
        \text{\cref{prop:nF} type} & $F$ & \textbf{${n_F}$} & \textbf{$\vcd(WF)$} & \textbf{$\lambda(F)$} \\ 
        &&&& \\[-0.5em]
        \hline
        (1) & $\mathbb{Z}/9$ & $3$ & $0$ & $2$ \\ 
            & $\mathbb{Z}/3$ & $5$ & $2$ & $1$ \\ 
        \hline
        (2.1) & $\mathbb{Z}/10$ & $3$ & $0$ & $2$ \\
              & $\mathbb{Z}/5$ & $4$ & $1$ & $1$ \\
              & $\mathbb{Z}/2$ & $7$ & $4$ & $1$ \\
        \hline
        (2.2) & $D_{2(10)}$ & $3$ & $0$ & $3$ \\
              & $D_{2(5)}$ & $3$ & $0$ & $2$ \\ 
              &            & $4$ & $1$ & $2$ \\ 
              & $D_{2(2)}$ & $5$ & $2$ & $2$ \\ 
        \hline
        (3.1) & $\mathbb{Z}/8$ & $3$ & $0$ & $3$ \\
              & $\mathbb{Z}/4$ & $4$ & $1$ & $2$ \\
              & $\mathbb{Z}/2$ & $6$ & $3$ & $1$ \\
        \hline
        (3.2) & $D_{2(8)}$ & $3$ & $0$ & $4$ \\
              & $D_{2(4)}$ & $3$ & $0$ & $3$ \\
              &            & $4$ & $1$ & $3$ \\
              & $D_{2(2)}$ & $4$ & $0$ & $2$ \\
              &            & $5$ & $2$ & $2$ \\
        \hline
    \end{tabular}
\end{table}}

{\footnotesize
\begin{table}[ht]
    \centering
    \caption{\small Cyclic and dihedral subgroups of $\Mod_0^{11}$.}
    \begin{tabular}{|c|c|c|c|c|}
        \hline 
        &&&& \\[-0.5em]
        \text{\cref{prop:nF} type} & $F$ & \textbf{${n_F}$} & \textbf{$\vcd(WF)$} & \textbf{$\lambda(F)$} \\ 
        &&&& \\[-0.5em]
        \hline
        (1) & $\mathbb{Z}/10$ & $3$ & $0$ & $2$ \\ 
            & $\mathbb{Z}/5$  & $4$ & $1$ & $1$ \\ 
            & $\mathbb{Z}/2$  & $7$ & $4$ & $1$ \\ 
        \hline
        (2.1) & $\mathbb{Z}/11$ & $3$ & $0$ & $1$ \\ 
        \hline
        (2.2) & $D_{2(11)}$ & $3$ & $0$ & $2$ \\ 
        \hline
        (3.1) & $\mathbb{Z}/9$  & $3$ & $0$ & $2$ \\ 
              & $\mathbb{Z}/3$  & $5$ & $2$ & $1$ \\ 
        \hline
        (3.2) & $D_{2(9)}$ & $3$ & $0$ & $2$ \\ 
              & $D_{2(3)}$ & $4$ & $1$ & $1$ \\ 
        \hline
    \end{tabular}
\end{table}}

{\footnotesize
\begin{table}[ht]
    \centering
    \caption{\small Cyclic and dihedral subgroups of $\Mod_0^{12}$.}
    \begin{tabular}{|c|c|c|c|c|}
        \hline 
        &&&& \\[-0.5em]
        \text{\cref{prop:nF} type} & $F$ & \textbf{${n_F}$} & \textbf{$\vcd(WF)$} & \textbf{$\lambda(F)$} \\ 
        &&&& \\[-0.5em]
        \hline
        (1) & $\mathbb{Z}/11$ & $3$ & $0$ & $1$ \\ 
        \hline
        (2.1) & $\mathbb{Z}/12$ & $3$ & $0$ & $3$ \\ 
        & $\mathbb{Z}/6$ & $4$ & $1$ & $2$ \\ 
        & $\mathbb{Z}/4$ & $5$ & $2$ & $2$ \\ 
        & $\mathbb{Z}/3$ & $6$ & $3$ & $1$ \\ 
        & $\mathbb{Z}/2$ & $8$ & $5$ & $1$ \\ 
        \hline
        (2.2) & $D_{2(12)}$ & $3$ & $0$ & $4$ \\ 
        & $D_{2(6)}$ & $3$ & $0$ & $3$ \\ 
        &  & $4$ & $1$ & $3$ \\ 
        & $D_{2(4)}$ & $4$ & $1$ & $3$ \\ 
        & $D_{2(3)}$ & $4$ & $1$ & $2$ \\ 
        &  & $5$ & $2$ & $2$ \\ 
        & $D_{2(2)}$ & $5$ & $2$ & $2$ \\ 
        &  & $6$ & $3$ & $2$ \\ 
        \hline
        (3.1) & $\mathbb{Z}/10$ & $3$ & $0$ & $2$ \\ 
        & $\mathbb{Z}/5$ & $4$ & $1$ & $1$ \\ 
        & $\mathbb{Z}/2$ & $7$ & $4$ & $1$ \\ 
        \hline
        (3.2) & $D_{2(10)}$ & $3$ & $0$ & $3$ \\ 
        & $D_{2(5)}$ & $3$ & $0$ & $2$ \\ 
        &  & $4$ & $1$ & $2$ \\ 
        & $D_{2(2)}$ & $5$ & $2$ & $2$ \\ 
        \hline
    \end{tabular}
\end{table}}

{\footnotesize
\begin{table}[ht]
    \centering
    \caption{\small Cyclic and dihedral subgroups of $\Mod_0^{13}$.}
    \begin{tabular}{|c|c|c|c|c|}
        \hline 
        &&&& \\[-0.5em]
        \text{\cref{prop:nF} type} & $F$ & \textbf{${n_F}$} & \textbf{$\vcd(WF)$} & \textbf{$\lambda(F)$} \\ 
        &&&& \\[-0.5em]
        \hline
        (1) & $\mathbb{Z}/12$ & $3$ & $0$ & $3$ \\ 
        & $\mathbb{Z}/6$ & $3$ & $0$ & $3$ \\ 
        & $\mathbb{Z}/4$ & $5$ & $2$ & $2$ \\ 
        & $\mathbb{Z}/2$ & $8$ & $5$ & $1$ \\ 
        \hline
        (2.1) & $\mathbb{Z}/13$ & $3$ & $0$ & $1$ \\ 
        \hline
        (2.2) & $D_{2(13)}$ & $3$ & $0$ & $2$ \\ 
        \hline
        (3.1) & $\mathbb{Z}/11$ & $3$ & $0$ & $1$ \\ 
        \hline
        (3.2) & $D_{2(11)}$ & $3$ & $0$ & $2$ \\ 
        \hline
    \end{tabular}
\end{table}}
\clearpage

\section{$\lambda(F)$ and $\vcd(WF)$ for finite subgroups of $\Mod_2^n$.}\label{appendixB}
In this appendix we use Broughton's classification \cite[Theorem 4.1 $\&$ Table 4]{FiniteGenus2} to give upper bounds for $n_F$, $\vcd(WF)$ and $\lambda(F)$ for all the conjugacy classes of non-trivial finite subgroups $F$ of $\Mod_2^n$ when $n\geq 0$.

{\footnotesize
\begin{table}[ht]
    \centering
    { \caption{\small Finite non-trivial subgroups of $\Mod_2^n$.}\label{Genus2}}
    \begin{tabular}{|c|c|c|c|c|c|}
        \hline
        &&&&&\\[-0.5em]
    $F$ & $\rvert F\rvert$ & $S_2/F$ & \textbf{${n_F}\leq$} & \textbf{$\vcd(WF)\leq$} & \textbf{$\lambda(F)\leq$}\\
    &&&&&\\[-0.5em]
        \hline
         &&&&&\\[-0.5em]
       $\mathbb{Z}/2$  & $2$ & $(S_0;2,2,2,2,2,2)$ & $\frac{n}{2}+6$ & $\frac{n}{2}+3$ & $1$ \\ 
       &&&&&\\[-0.5em]
       \hline
 &&&&&\\[-0.5em]
       $\mathbb{Z}/2$  & $2$ & $(S_1;2,2)$ & $\frac{n}{2}+2$ & $\frac{n}{2}+2$ & $1$ \\ 
       &&&&&\\[-0.5em]
       \hline
       &&&&&\\[-0.5em]
       $\mathbb{Z}/3$  & $3$ & $(S_0;3,3,3,3)$ & $\frac{n}{3}+4$ & $\frac{n}{3}+1$ & $1$ \\ 
       &&&&&\\[-0.5em]
       \hline
       &&&&&\\[-0.5em]
        $\ \mathbb{Z}/2\times\mathbb{Z}/2\ $  & $4$ & $(S_0;2,2,2,2,2)$ & $\frac{n}{4}+5$ & $\frac{n}{4}+2$ & $2$ \\ 
        &&&&&\\[-0.5em]
       \hline
       &&&&&\\[-0.5em]
       $\mathbb{Z}/4$  & $4$ & $(S_0;2,2,4,4)$ & $\frac{n}{4}+4$ & $\frac{n}{4}+1$ & $2$ \\ 
       &&&&&\\[-0.5em]
       \hline
       &&&&&\\[-0.5em]
        $\mathbb{Z}/5$  & $5$ & $(S_0;5,5,5)$ & $\frac{n}{5}+3$ & $\frac{n}{5}$ & $1$ \\ 
        &&&&&\\[-0.5em]
       \hline
       &&&&&\\[-0.5em]
        $\mathbb{Z}/6$  & $6$ & $(S_0;3,6,6)$ & $\frac{n}{6}+3$ & $\frac{n}{6}$ & $2$ \\ 
        &&&&&\\[-0.5em]
       \hline
       &&&&&\\[-0.5em]
         $\mathbb{Z}/6$  & $6$ & $(S_0;2,2,3,3)$ & $\frac{n}{6}+4$ & $\frac{n}{6}$+1 & $2$ \\ 
         &&&&&\\[-0.5em]
       \hline
       &&&&&\\[-0.5em]
          $D_{2(3)}$  & $6$ & $(S_0;2,2,3,3)$ & $\frac{n}{6}+4$ & $\frac{n}{6}$+1 & $2$ \\ 
          &&&&&\\[-0.5em]
       \hline
       &&&&&\\[-0.5em]
     $\mathbb{Z}/8$   & $8$ & $(S_0;2,8,8)$ & $\frac{n}{8}+3$ & $\frac{n}{8}$ & $3$ \\ 
     &&&&&\\[-0.5em]
       \hline
       &&&&&\\[-0.5em]
      $\widetilde{D_2}$   & $8$ & $(S_0;4,4,4)$ & $\frac{n}{8}+3$ & $\frac{n}{8}$ & $2$ \\ 
       &&&&&\\[-0.5em]
       \hline
        &&&&&\\[-0.5em]
          $D_{2(4)}$  & $8$ & $(S_0;2,2,2,4)$ & $\frac{n}{8}+4$ & $\frac{n}{8}$+1 & $3$ \\ 
          &&&&&\\[-0.5em]
       \hline
            &&&&&\\[-0.5em]
        $\mathbb{Z}/10$   & $10$ & $(S_0;2,5,10)$ & $\frac{n}{10}+3$ & $\frac{n}{10}$ & $2$ \\ 
     &&&&&\\[-0.5em]
       \hline
           &&&&&\\[-0.5em]
        $\ \mathbb{Z}/2\times\mathbb{Z}/6\ $  & $12$ & $(S_0;2,6,6)$ & $\frac{n}{12}+3$ & $\frac{n}{12}$ & $3$ \\ 
         &&&&&\\[-0.5em]
        \hline
        &&&&&\\[-0.5em]
           $D_{4,3,-1}$  & $12$ & $(S_0;3,4,4)$ & $\frac{n}{12}+3$ & $\frac{n}{12}$ & $3$ \\ 
        &&&&&\\[-0.5em]
       \hline
        &&&&&\\[-0.5em]
          $D_{2(6)}$  & $12$ & $(S_0;2,2,2,3)$ & $\frac{n}{12}+4$ & $\frac{n}{12}+1$ & $3$ \\ 
          &&&&&\\[-0.5em]
       \hline
         &&&&&\\[-0.5em]
           $D_{2,8,3}$  & $16$ & $(S_0;2,4,8)$ & $\frac{n}{16}+3$ & $\frac{n}{16}$ &  $4$ \\ 
          &&&&&\\[-0.5em]
       \hline
          &&&&&\\[-0.5em]
         $\ \mathbb{Z}/2\rtimes (\mathbb{Z}/2\times\mathbb{Z}/2\times \mathbb{Z}/3)\ $  & $24$ & $(S_0;2,4,6)$ & $\frac{n}{24}+3$ & $\frac{n}{24}$ &  $4$ \\ 
          &&&&&\\[-0.5em]
       \hline
         &&&&&\\[-0.5em]
         $\text{SL}_2(3)$ & $24$ & $(S_0;3,3,4)$ & $\frac{n}{24}+3$ & $\frac{n}{24}$ &  $4$ \\ 
          &&&&&\\[-0.5em]
       \hline
        &&&&&\\[-0.5em]
         $\text{GL}_2(4)$ & $48$ & $(S_0;2,3,8)$ & $\frac{n}{48}+3$ & $\frac{n}{48}$ & $5$ \\ 
          &&&&&\\ 
       \hline
       
        \end{tabular}
\end{table}}

\clearpage

%% file: main.bbl
\begin{thebibliography}{JRLASS24}

\bibitem[AJPTN18]{AJPTN18}
J.~Aramayona, D.~Juan-Pineda, and A.~Trujillo-Negrete.
\newblock On the virtually cyclic dimension of mapping class groups of punctured spheres.
\newblock {\em Algebr. Geom. Topol.}, 18(4):2471--2495, 2018.

\bibitem[AMP14]{AMP14}
J.~Aramayona and C.~Mart\'{\i}nez-P\'{e}rez.
\newblock The proper geometric dimension of the mapping class group.
\newblock {\em Algebr. Geom. Topol.}, 14(1):217--227, 2014.

\bibitem[BG05]{bellingeri}
P.~Bellingeri and E.~Godelle.
\newblock {Questions on surface braid groups}.
\newblock working paper or preprint, March 2005.

\bibitem[Bir74]{BIRMAN}
J.~S. Birman.
\newblock {\em Braids, links, and mapping class groups}.
\newblock Annals of Mathematics Studies, No. 82. Princeton University Press, Princeton, N.J.; University of Tokyo Press, Tokyo, 1974.

\bibitem[BLN01]{BLN01}
N.~Brady, I.~J. Leary, and B.~E.~A. Nucinkis.
\newblock On algebraic and geometric dimensions for groups with torsion.
\newblock {\em J. London Math. Soc. (2)}, 64(2):489--500, 2001.

\bibitem[Bro91]{FiniteGenus2}
S.~A. Broughton.
\newblock Classifying finite group actions on surfaces of low genus.
\newblock {\em J. Pure Appl. Algebra}, 69(3):233--270, 1991.

\bibitem[CJLS24]{LaEspinayLos4}
Nestor {Colin}, Rita {Jim{\'e}nez Rolland}, Porfirio~L. {Le{\'o}n {\'A}lvarez}, and Luis~Jorge {S{\'a}nchez Salda{\~n}a}.
\newblock {On the dimension of Harer's spine for the decorated Teichm{\"u}ller space}.
\newblock {\em arXiv e-prints}, page arXiv:2409.04392, September 2024.

\bibitem[CK86]{CK86}
Frank Connolly and Tadeusz Ko\'{z}niewski.
\newblock Finiteness properties of classifying spaces of proper {$\Gamma$}-actions.
\newblock {\em J. Pure Appl. Algebra}, 41(1):17--36, 1986.

\bibitem[FM12]{FM12}
B.~Farb and D.~Margalit.
\newblock {\em A primer on mapping class groups}, volume~49 of {\em Princeton Mathematical Series}.
\newblock Princeton University Press, Princeton, NJ, 2012.

\bibitem[GcG17]{VcdBraid}
D.~L. Gon\c~calves and J.~Guaschi.
\newblock Inclusion of configuration spaces in {C}artesian products, and the virtual cohomological dimension of the braid groups of {$\mathbb{S}^2$} and {$\mathbb{R}P^2$}.
\newblock {\em Pacific J. Math.}, 287(1):71--99, 2017.

\bibitem[GJP15]{BraidSurvey}
J.~Guaschi and D.~Juan-Pineda.
\newblock A survey of surface braid groups and the lower algebraic {$K$}-theory of their group rings.
\newblock In {\em Handbook of group actions. {V}ol. {II}}, volume~32 of {\em Adv. Lect. Math. (ALM)}, pages 23--75. Int. Press, Somerville, MA, 2015.

\bibitem[Har86]{Ha86}
J.~L. Harer.
\newblock The virtual cohomological dimension of the mapping class group of an orientable surface.
\newblock {\em Invent. Math.}, 84(1):157--176, 1986.

\bibitem[HSSTN22]{HSST}
C.~E. Hidber, L.~J. S\'anchez~Salda{\~{n}}a, and A.~Trujillo-Negrete.
\newblock On the dimensions of mapping class groups of non-orientable surfaces.
\newblock {\em Homology Homotopy Appl.}, 24(1):347--372, 2022.

\bibitem[Ji14]{Ji14}
L.~Ji.
\newblock Well-rounded equivariant deformation retracts of {T}eichm\"uller spaces.
\newblock {\em Enseign. Math.}, 60(1-2):109--129, 2014.

\bibitem[JPTN16]{JPT16}
D.~Juan-Pineda and A.~Trujillo-Negrete.
\newblock On classifying spaces for the family of virtually cyclic subgroups in mapping class groups.
\newblock {\em Pure Appl. Math. Q.}, 12(2):261--292, 2016.

\bibitem[JRLASS24]{JRLASS24}
R.~Jim\'enez~Rolland, P.~L. Le\'on~\'Alvarez, and L.~J. S\'anchez~Salda{\~n}a.
\newblock Commensurators of abelian subgroups and the virtually abelian dimension of mapping class groups.
\newblock {\em J. Pure Appl. Algebra}, 228(6):Paper No. 107566, 22, 2024.

\bibitem[JW10]{WolpertJi}
L.~Ji and S.~A. Wolpert.
\newblock A cofinite universal space for proper actions for mapping class groups.
\newblock In {\em In the tradition of {A}hlfors-{B}ers. {V}}, volume 510 of {\em Contemp. Math.}, pages 151--163. Amer. Math. Soc., Providence, RI, 2010.

\bibitem[Ker83]{KER}
S.~P. Kerckhoff.
\newblock The {N}ielsen realization problem.
\newblock {\em Ann. of Math. (2)}, 117(2):235--265, 1983.

\bibitem[LGGM18]{MR3869010}
D.~Lima~Gon{\c{c}}alves, J.~Guaschi, and M.~Maldonado.
\newblock Embeddings and the (virtual) cohomological dimension of the braid and mapping class groups of surfaces.
\newblock {\em Confluentes Math.}, 10(1):41--61, 2018.

\bibitem[L{\"u}c89]{Lu89}
W.~L{\"u}ck.
\newblock {\em Transformation groups and algebraic {$K$}-theory}, volume 1408 of {\em Lecture Notes in Mathematics}.
\newblock Springer-Verlag, Berlin, 1989.
\newblock Mathematica Gottingensis.

\bibitem[L{\"u}c05]{Lu05}
W.~L{\"u}ck.
\newblock Survey on classifying spaces for families of subgroups.
\newblock In {\em Infinite groups: geometric, combinatorial and dynamical aspects}, volume 248 of {\em Progr. Math.}, pages 269--322. Birkh\"auser, Basel, 2005.

\bibitem[Mah11]{Ma11}
J.~Maher.
\newblock Random walks on the mapping class group.
\newblock {\em Duke Math. J.}, 156(3):429--468, 2011.

\bibitem[MV03]{MV03}
G.~Mislin and A.~Valette.
\newblock {\em Proper group actions and the {B}aum-{C}onnes conjecture}.
\newblock Advanced Courses in Mathematics. CRM Barcelona. Birkh\"auser Verlag, Basel, 2003.

\bibitem[NP18]{Nucinkis:Petrosyan}
B.~Nucinkis and N.~Petrosyan.
\newblock Hierarchically cocompact classifying spaces for mapping class groups of surfaces.
\newblock {\em Bull. Lond. Math. Soc.}, 50(4):569--582, 2018.

\bibitem[Stu04]{STUKOW}
M.~Stukow.
\newblock Conjugacy classes of finite subgroups of certain mapping class groups.
\newblock {\em Turkish J. Math.}, 28(2):101--110, 2004.

\bibitem[Thu22]{ThurstonBook}
W.~P. Thurston.
\newblock {\em The geometry and topology of three-manifolds. {V}ol. {IV}}.
\newblock American Mathematical Society, Providence, RI, [2022].
\newblock Edited and with a preface by Steven P. Kerckhoff and a chapter by J. W. Milnor.

\end{thebibliography}
